\documentclass[11pt]{article}
\usepackage{amsmath, amssymb, amsthm, amsfonts}
\usepackage{indentfirst}
\usepackage[mathscr]{euscript}
\usepackage{amscd}
\usepackage{color}
\usepackage[pdftex]{graphicx}

\numberwithin{equation}{section}
\numberwithin{figure}{section}

\theoremstyle{plain}
\newtheorem{theorem}{Theorem}\numberwithin{theorem}{section}
\newtheorem{lemma}{Lemma}\numberwithin{lemma}{section}
\numberwithin{proposition}{section}
\newtheorem{corollary}{Corollary}\numberwithin{corollary}{section}
\theoremstyle{definition}
\numberwithin{definition}{section}
\theoremstyle{remark}
\newtheorem{remark}{Remark}\numberwithin{remark}{section}

\newcommand{\R}{\mathbb{R}}

\newcommand{\J}{\mathbb{J}}
\newcommand{\F}{{}_1F_2}

\title{Rational extension of Newton diagram for\\ the positivity of
$\F$ hypergeometric functions\\ and Askey-Szeg\"o problem}

\author{Yong-Kum Cho, Seok-Young Chung and Hera Yun}

\begin{document}
\maketitle

{\bf Abstract.} We present a rational extension of Newton diagram for the positivity
of $\F$ generalized hypergeometric functions. As an application, we give upper and lower bounds for
the transcendental roots $\beta(\alpha)$ of
\begin{align*}
\int_0^{j_{\alpha, 2}} t^{-\beta} J_\alpha(t) dt = 0\qquad(-1<\alpha\le 1/2),
\end{align*}
where $j_{\alpha, 2}$ denotes the second positive zero of Bessel function $J_\alpha$.

\bigskip

{\bf Keywords.} {\small Bessel function, generalized hypergeometric function, Newton diagram, positivity,
Saalsch\"utzian.}

\bigskip

%\noindent
{\small {\bf  2010 Mathematics Subject Classification:} 26D15, 33C10, 33C20.}

\section{Introduction}
We consider the problem of determining $(\alpha, \beta)$ for
\begin{equation}\label{AG1}
\int_0^x t^{-\beta} J_\alpha(t)dt\ge 0\qquad(x>0),
\end{equation}
where $J_\alpha$ stands for the first-kind Bessel function of order $\alpha$. For the sake of convergence and application,
it will be assumed $\,\alpha>-1, \,\beta<\alpha+1.$

Owing to various applications, the problem has been studied by many authors over a long period of time.
In connection with the monotonicity of Bessel functions, for instance, the problem dates back to Bailey \cite{Ba2} and
Cooke \cite{C}. We refer to Askey \cite{As1}, \cite{As2} for further historical backgrounds.

By interpolating known results for some special cases in certain way, Askey \cite{As2} described an explicit range of
parameters as follows.

\medskip

\noindent
{\bf Theorem A.}\hskip.2cm Let $\mathcal{P}$ be the set of $\,(\alpha, \beta)\in\R^2\,$ defined by
$$\mathcal{P} =\bigl\{\alpha>-1,\, \,0\le\beta<\alpha+1\bigr\}
\cup\left\{\alpha\ge 0,\,\,\max\left(-\alpha,\,-\frac 12\right)\le\beta\le 0\right\}.$$
\begin{itemize}
\item[(i)] For each $\,(\alpha, \beta)\in\mathcal{P},\,$ the inequality of \eqref{AG1} holds with strict positivity
unless it coincides with $\,(1/2, -1/2).$
\item[(ii)] If $\,\alpha>-1, \,\beta<-1/2\,,$ then \eqref{AG1} does not hold.
\end{itemize}

\begin{figure}[!h]
 \centering
 \includegraphics[width=300pt, height=220pt]{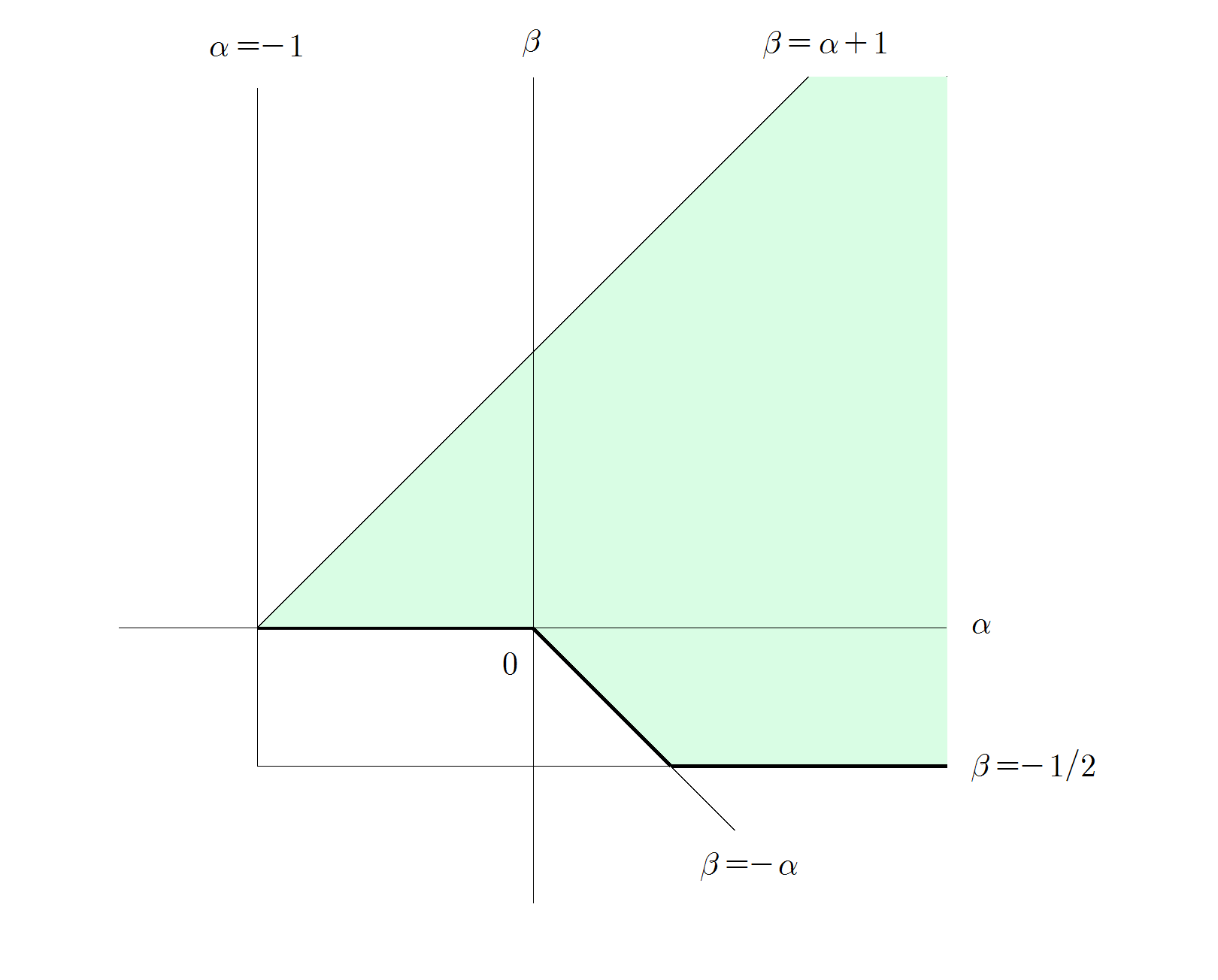}
 \caption{The known positivity region for problem \eqref{AG1}.}
\label{Fig1}
\end{figure}

As it is shown in Figure 1, the positivity region $\mathcal{P}$ represents an infinite
polygon enclosed by four boundary lines
$$\beta=\alpha+1, \,\,\beta=0, \,\,\beta=-\alpha, \,\,\beta=-1/2.$$
By part (ii), observed by Steinig \cite{St}, Theorem A leaves only the trapezoid
\begin{equation}\label{AG2}
\mathcal{T}=\left\{-1<\alpha<\frac 12,\,\,-\frac 12\le \beta<\min\left(0, \,\,-\alpha\right)\right\}
\end{equation}
undetermined in regards to problem \eqref{AG1}.

As for this missing region, the best possible range of parameters is known in an implicit formulation
which involves roots of certain transcendental equations. To be precise, we follow Askey's summary \cite{As2} to state

\medskip

\noindent
{\bf Theorem B.}\hskip.2cm Let $j_{\alpha, 2}\,$ be the second positive zero of $J_\alpha(t),\,\alpha>-1.$
\begin{itemize}
\item[(i)] For $\,-1<\alpha\le 1/2,\,$ \eqref{AG1} holds if and only if
$\,\beta\ge\beta(\alpha),\,$ where $\beta(\alpha)$ denotes the unique zero of
\begin{equation}\label{AG3}
A(\beta) =\int_0^{j_{\alpha, 2}} t^{-\beta} J_\alpha(t) dt,\quad -\frac 12<\beta<\alpha+1.
\end{equation}
\item[(ii)] As a special case of \eqref{AG1}, the inequality
\begin{equation}\label{AG4}
\int_0^x t^{-\alpha} J_\alpha(t)dt \ge 0\qquad(x>0)
\end{equation}
holds for $\,\alpha\ge\bar{\alpha},\,$ where $\bar{\alpha}$ denotes the unique zero of
\begin{equation}\label{AG5}
G(\alpha) = \int_0^{j_{\alpha, 2}} t^{-\alpha} J_\alpha(t) dt,\quad \alpha>-\frac 12.
\end{equation}
\end{itemize}

\medskip

Regarding part (i), the existence and uniqueness of such a zero as well as the positivity of \eqref{AG1}
is due to Makai \cite{Ma1}, \cite{Ma2} when $\,-1/2<\alpha<1/2,\,$ Askey and Steinig \cite{AS} when $\,-1<\alpha<-1/2,\,$ respectively.
The remaining case $\,\alpha=\pm 1/2\,$ follows by an integration by parts.

Part (ii) is obtained by G. Szeg\"o \cite{Fe} much earlier and reproved by Koumandos \cite{K}, Lorch, Muldoon and P. Szeg\"o \cite{LMS}.
Since \eqref{AG4} is a special case of \eqref{AG1} and part (i) gives a necessary and sufficient condition
for \eqref{AG1}, it is equivalent to define $\bar{\alpha}$ as the unique
solution of $\,\beta(\alpha)=\alpha.$

A major drawback of Theorem B lies in the intricate nature of the zeros $\beta(\alpha)$
and $\bar{\alpha}$. As it is pointed out by Askey \cite{As2}, in fact,
essentially nothing has been known yet on the nature of $\beta(\alpha)$ and $\bar{\alpha}$ except a few
numerical simulations and trivial limiting behavior $\,\lim_{\alpha\to -1 +} \beta(\alpha) = 0.$

In this paper we aim at extending the positivity region $\mathcal{P}$ of Theorem A
and thereby obtaining informative bounds of $\beta(\alpha)$ and $\bar{\alpha}$ which provide
an insight into their nature and an approximating means in practical use.

By making use of the identity (Luke \cite{L}, Watson \cite{Wa})
\begin{equation}\label{SF1}
J_\alpha(t) = \frac{1}{\Gamma(\alpha+1)} \left(\frac t2\right)^\alpha{}_0F_1\left(\alpha+1\,;-\frac{t^2}{4}\right)
\qquad(\alpha>-1)
\end{equation}
and integrating termwise, it is easy to see
\begin{align}\label{AG6}
\int_0^x t^{-\beta}J_\alpha(t) dt &= \frac{x^{\alpha-\beta+1}}{2^\alpha(\alpha-\beta+1)\Gamma(\alpha+1)}\nonumber\\
&\qquad\times\, \F\left[\begin{array}{c} \frac{\alpha-\beta+1}{2}\\ \alpha+1,
\frac{\alpha-\beta+3}{2}\end{array} \biggl| -\frac{\,x^2}{4}\right]
\end{align}
and hence problem \eqref{AG1} is equivalent to the problem of positivity for the functions
defined on the right side of \eqref{AG6}.

More generally, we shall be concerned with the positivity of generalized
hypergeometric functions of type
\begin{equation}\label{AG7}
\F\left[\begin{array}{c}
a\\ b, c\end{array}\biggr| -\frac{\,x^2}{4}\right]\qquad(x>0)
\end{equation}
with parameters $\,a>0, \,b>0, \,c>0.$
In the recent work \cite{CY}, to be explained in detail, a positivity criterion
for the functions of type \eqref{AG7} is established in terms of the Newton diagram associated to
$\,\{(a+1/2, 2a), \,(2a, a+1/2)\}.\,$ Due to certain region of parameters left undetermined, however, it turns out
that an application of the criterion to \eqref{AG6} yields Theorem A immediately
but does not cover the missing region $\mathcal{T}$ either.

The main purpose of this paper is to give an extension of the Newton diagram
which leads to an improvement of Theorem A in an explicit way and provides information on the
nature of $\beta(\alpha)$ and $\bar{\alpha}$.

As it is more or less standard in the theory of special functions,
we shall carry out Gasper's {\it sums of squares method} \cite{Ga1}
for investigating positivity,
which essentially reduces the matter to how to determine the signs of ${}_4F_3$ terminating series given in the form
\begin{equation}\label{AG8}
{}_{4}F_{3} \left[\begin{array}{c} -n, n+\alpha_1, \alpha_2, \alpha_3\\
\beta_1, \beta_2, \beta_{3}\end{array}\right], \quad n=1, 2, \cdots,
\end{equation}
for appropriate values of $\,\alpha_j, \beta_j\,$ expressible in terms of $\,a, b, c.\,$

From a technical point of view, if we express \eqref{AG8} as a finite sum with index $k$,
it is the alternating factor $(-n)_k$ that causes main difficulties in analyzing its sign.
To circumvent, we shall apply Whipple's transformation formula
to convert it into a ${}_7F_6$ terminating series which does not involve such an alternating factor.
By estimating a lower bound for the transformed series,
we shall deduce positivity in an inductive way.

While Askey and Szeg\"o studied problem \eqref{AG1} primarily as a limiting case for the positivity of
certain sums of Jacobi polynomials, there are many other applications and generalizations (see e.g.
\cite{Fe, Ga1, Ga2, MR}). As an exemplary generalization, we shall consider
the integrals of type
$$\int_0^x (x^2- t^2)^\gamma t^{-\beta}J_\alpha(t) dt\qquad(x>0),$$
and obtain the range of parameters for its positivity by applying our new criterion, which improves the work
of Gasper \cite{Ga1} considerably.

\section{Preliminaries}
As it is standard, given nonnegative integers $p, q$, we shall define and write ${}_pF_q$
generalized hypergeometric functions in the form
\begin{equation}\label{A1}
{}_pF_q\left[\begin{array}{c} \alpha_1, \cdots, \alpha_p\\
\beta_1, \cdots, \beta_q\end{array}\biggr| \,z\right]
= \sum_{k=0}^\infty \frac{(\alpha_1)_k\cdots (\alpha_p)_k}{k!\,(\beta_1)_k
\cdots (\beta_q)_k}\, z^k\qquad(z\in\mathbb{C}),
\end{equation}
where the coefficients are written in Pochhammer's notation, that is,
for any $\,\alpha\in\R,\,$ $\,(\alpha)_k = \alpha(\alpha+1)\cdots(\alpha+k-1)\,$ when $\,k\ge 1\,$
and $\,(\alpha)_0=1.$ In the case when $\,z=1,\,$ we shall delete
the argument $z$ in what follows.

A function of type \eqref{A1} is said to be {\it Saalsch\"utzian}
when the parameters satisfy the condition
$\,1+\alpha_1+\cdots+\alpha_p = \beta_1+\cdots +\beta_q.\,$
If one of the numerator-parameters $\alpha_j$ is a negative integer, e.g., $\,\alpha_1=-n\,$
with $n$ a positive integer, then it becomes a {\it terminating series} given by
\begin{equation}\label{A2}
{}_pF_q\left[\begin{array}{c} -n, \alpha_2, \cdots, \alpha_p\\
\beta_1, \cdots, \beta_q\end{array}\biggr| \,z\right]
= \sum_{k=0}^n (-1)^k \binom nk \frac{(\alpha_2)_k\cdots (\alpha_p)_k}{(\beta_1)_k \cdots (\beta_q)_k}\, z^k.
\end{equation}

For the generalized hypergeometric functions of type \eqref{A1} which are both terminating and Saalsch\"utzian,
there are a number of formulas available for summing or transforming into other terminating series.
Of particular importance will be the following extracted from Bailey \cite{Ba}.

\begin{itemize}
\item[(i)] (Saalsch\"utz's formula, \cite [2.2(1)] {Ba})

If $\,1+\alpha_1+\alpha_2=\beta_1+\beta_2,\,$ then
\begin{align}\label{A3}
{}_3F_2\left[\begin{array}{c} -n, n+\alpha_1, \alpha_2
\\{} \beta_1, \beta_2\end{array}\right] =
\frac{\left(\beta_1-\alpha_2\right)_n \left(\beta_2-\alpha_2\right)_n}
{\left(\beta_1\right)_n\left(\beta_2\right)_n}\,.
\end{align}
\item[(ii)] (Whipple's transformation formula, \cite[4.3(4)]{Ba})

If $\,1+\alpha_1+\alpha_2 +\alpha_3= \beta_1 +\beta_2 +\beta_3,\,$
then
\begin{align}\label{A4}
{}_{4}F_{3} &\left[\begin{array}{c} -n, n+\alpha_1, \alpha_2, \alpha_3\\
{} \beta_1, \beta_2, \beta_{3}\end{array}\right]
=\frac{(1+\alpha_1-\beta_3)_n (\beta_3-\alpha_2)_n}
{(1+\sigma)_n (\beta_3)_n}\quad\times\nonumber\\
&\qquad {}_7F_6\biggl[\begin{array}{cccc} \sigma, &1+ \sigma/2, &-n, &n+\alpha_1,\\
&\sigma/2, &n+1+\sigma, &-n+1+\alpha_2-\beta_3,\end{array}\nonumber\\
&\qquad\qquad\qquad\qquad\begin{array}{ccc} \alpha_2, &\beta_1-\alpha_3, &\beta_2-\alpha_3\\
1+\alpha_1-\beta_3, &\beta_2, &\beta_1
\end{array}\biggr],
\end{align}
where we put $\,\sigma=\alpha_1+\alpha_2-\beta_3.\,$ It is a modification of the original
form suitable to the present application and arranged in such a way that the sums of columns of the ${}_6F_6$ terminating series,
obtained from deleting $\sigma$,
are all equal to $1+\sigma$.
\end{itemize}

\section{Positivity of ${}_4F_3$ terminating series}
The purpose of this section is to prove the following positivity result for
a special class of terminating ${}_4F_3$ generalized hypergeometric series, which will be crucial
in our subsequent developments.

\smallskip

\begin{lemma}\label{lemmaA2}
For each positive integer $n$, put
$$\Theta_n={}_{4}F_{3} \left[\begin{array}{c} -n, n+\alpha_1, \alpha_2, \alpha_3\\
\beta_1, \beta_2, \beta_{3}\end{array}\right].$$
Suppose that $\,\alpha_j, \beta_j\,$ satisfy the following assumptions simultaneously:
\begin{equation}\label{A}
\left\{\begin{aligned}
&{{\rm(A1)}\quad 1+\alpha_1+\alpha_2+\alpha_3 = \beta_1+\beta_2+\beta_3,}\\
&{{\rm(A2)}\quad 0<\alpha_2 <\beta_3\le 2+\alpha_1,}\\
&{{\rm(A3)}\quad 0<\alpha_3 <\min \bigl(\beta_1,\,\beta_2\bigr),}\\
&{{\rm(A4)}\quad (1+\alpha_1)\alpha_2\alpha_3 \le \beta_1\beta_2\beta_3.}\end{aligned}\right.
\end{equation}
Then $\,\Theta_1\ge 0\,$ and $\,\Theta_n >0\,$ for all $\,n\ge 2.$
\end{lemma}

\begin{proof}
We apply Whipple's transformation formula to transform $\Theta_n$ into
a product of ${}_3F_2$ and ${}_7F_6$ terminating series as stated in \eqref{A4}.
By using
$$(\alpha)_n = (\alpha)_k (k+\alpha)_{n-k}\,,\quad (\alpha)_n= (\alpha)_{n-k}(n-k+\alpha)_k, $$
valid for any real number $\alpha$ and $\,k=0, \cdots, n,\,$ it is equivalent to
\begin{align}\label{A5}
\Theta_n &= \frac{1}
{(1+\sigma)_n (\beta_3)_n}\,\Omega_n\,,\nonumber\\
\Omega_n &=\sum_{k=0}^n \binom nk (k+1+\alpha_1-\beta_3)_{n-k}(\beta_3-\alpha_2)_{n-k}\,\frac{
(n+\alpha_1)_k}{(n+1+\sigma)_k}
\nonumber\\
&\qquad\times\,\, \frac{(\alpha_2)_k (\beta_1-\alpha_3)_k(\beta_2-\alpha_3)_k}{(\beta_1)_k (\beta_2)_k}
\frac{(\sigma)_k (1+\sigma/2)_k}{(\sigma/2)_k},
\end{align}
where $\,\sigma = \alpha_1+\alpha_2-\beta_3\,$ and the last factor must be understood as
\begin{equation}\label{A5-1}
\frac{(\sigma)_k (1+\sigma/2)_k}{(\sigma/2)_k} = \left\{
\begin{aligned} &{\qquad 1} &{\text{for}\quad k=0,}\\
&{(1+\sigma)_{k-1}(2k+\sigma)} &{\text{for}\quad k\ge 1.}
\end{aligned}\right.
\end{equation}

By the Saalsch\"utzian condition of (A1) and (A3), we observe that
$$\,1+\sigma=\beta_1+\beta_2-\alpha_3>0\,$$ and hence the
positivity or nonnegativity of $\Theta_n$ reduces to that of $\Omega_n$. We also note that the assumptions of
(A1), (A2), (A3) imply
$$\,1+\alpha_1= \beta_1+ (\beta_2 -\alpha_3)+ (\beta_3-\alpha_2)>0.$$ As a consequence, if
$\,\beta_3\le 1+\alpha_1,\,$ then the first term is nonnegative and all of other terms are positive so that
$\,\Omega_n>0\,$ for each $\,n\ge 1.$ Therefore it suffices to deal with the case
$\,\beta_3>1+\alpha_1,\,$ which will be assumed hereafter.

In the special case $\,n=1,\,$
it is a matter of algebra to factor out
\begin{align}\label{A6}
\Omega_1 &= (1+\alpha_1-\beta_3)(\beta_3-\alpha_2) + \frac{(1+\alpha_1)\alpha_2(\beta_1-\alpha_3)(\beta_2-\alpha_3)}
{\beta_1\beta_2}\nonumber\\
 &=\frac{(\beta_1+\beta_2-\alpha_3)\bigl[\beta_1\beta_2\beta_3-(1+\alpha_1)\alpha_2\alpha_3\bigr]}
{\beta_1\beta_2}
\end{align}
which clearly shows $\,\Omega_1\ge 0\,$ under the stated assumptions.

For $\,n\ge 2,\,$ we shall deduce the strict positivity of $\Omega_n$ by considering
each case $\,\beta_3\ge 1+\alpha_2,\,\beta_3<1+\alpha_2\,$ separately
in the following manner.

\bigskip

{\bf I. The case $\,\beta_3\ge 1+\alpha_2.$} \hskip.2cm We claim that
\begin{equation}\label{A7}
\Omega_n>(2+\alpha_1-\beta_3)_{n-1}(1+\beta_3-\alpha_2)_{n-1}\Omega_1.
\end{equation}

To verify, we observe that each term of $\Omega_n$ except the first one
is positive so that $\Omega_n$ exceeds the sum of the first two terms, which implies
\begin{align}\label{A7-1}
\Omega_n &> (2+\alpha_1-\beta_3)_{n-1}(\beta_3-\alpha_2)_n\nonumber\\
&\qquad\times \biggl[ 1+\alpha_1-\beta_3 +\frac{n(n+\alpha_1)\alpha_2(\beta_1-\alpha_3)(\beta_2-\alpha_3)(2+\sigma)}
{(n+1+\sigma)(n-1+\beta_3-\alpha_2)\beta_1\beta_2}\biggr].
\end{align}

If we set
\begin{align*}
f(n) &= \frac{n(n+\alpha_1)}{(n+1+\sigma)(n-1+\beta_3-\alpha_2)}\\
&= \frac{n^2 + \alpha_1n}{n^2 + \alpha_1 n + (1+\sigma)(\beta_3-1-\alpha_2)}
\end{align*}
and regard $n$ as a continuous variable, then the derivative of $f$ is given by
\begin{equation*}
f'(n) = \frac{(2n+\alpha_1)(1+\sigma)(\beta_3-1-\alpha_2)}
{\left[n^2 + \alpha_1 n + (1+\sigma)(\beta_3-1-\alpha_2)\right]^2}.
\end{equation*}
Due to the case assumption, it shows $\,f'(n)\ge 0\,$
on the interval $[1, \infty)$ and hence we may conclude $\,f(n)\ge f(1),\,$ that is,
$$ \frac{n(n+\alpha_1)}{(n+1+\sigma)(n-1+\beta_3-\alpha_2)}\ge \frac{1+\alpha_1}{(2+\sigma)(\beta_3-\alpha_2)}.$$

Reflecting this estimate in \eqref{A7-1} and simplifying, we obtain
\begin{align*}
\Omega_n &> (2+\alpha_1-\beta_3)_{n-1}\frac{(\beta_3-\alpha_2)_n}{(\beta_3-\alpha_2)}\\
&\qquad\times\,\biggl[(1+\alpha_1-\beta_3)(\beta_3-\alpha_2)
+\frac{(1+\alpha_1)\alpha_2(\beta_1-\alpha_3)(\beta_2-\alpha_3)}
{\beta_1\beta_2}\biggr]\\
&=(2+\alpha_1-\beta_3)_{n-1}(1+\beta_3-\alpha_2)_{n-1}\Omega_1,
\end{align*}
which proves \eqref{A7}. The strict positivity of $\Omega_n$ is an immediate consequence
of this inequality and the nonnegativity of $\Omega_1$.

\bigskip

{\bf II. The case $\,\beta_3<1+\alpha_2.$} \hskip.2cm In this case, we shall deduce the strict
positivity of $\Omega_n$ by induction on $n$. To simplify notation, we put
\begin{align*}
A_{n, k} &= \binom nk (k+1+\alpha_1-\beta_3)_{n-k}(\beta_3-\alpha_2)_{n-k}\,\frac{
(n+\alpha_1)_k}{(n+1+\sigma)_k},\\
B_k &=\frac{(\alpha_2)_k (\beta_1-\alpha_3)_k(\beta_2-\alpha_3)_k}{(\beta_1)_k (\beta_2)_k}
\frac{(\sigma)_k (1+\sigma/2)_k}{(\sigma/2)_k}
\end{align*}
so that $\,\Omega_n = \sum_{k=0}^n A_{n, k} B_k\,.$ By the stated assumptions
and \eqref{A5-1}, we note that
$\,A_{n, 0}<0\,$ but $\,A_{n, 1}\ge 0,\,\, A_{n, k}>0\,$ for $\,2\le k\le n\,$ and
$\,B_k>0\,$ for each $\,0\le k\le n.$ In this notation we claim that
\begin{equation}\label{A8}
\Omega_{n+1}>A_{n+1, n+1} B_{n+1} +  \left[\frac{(n+1)(n+1+\alpha_1-\beta_3)(n+1+\sigma)}{n+\alpha_1}\right]\Omega_n.
\end{equation}
As it is already shown that $\,\Omega_1\ge 0,\,$ once \eqref{A8} were true, it follows by an obvious induction argument that
we may conclude $\,\Omega_n>0\,$ for all $\,n\ge 2\,$ .

To verify, we make use of the identities
\begin{align*}
\binom{n+1}{k} &=\binom nk \,\frac{n+1}{n+1-k}\,,\\
(k+1+\alpha_1-\beta_3)_{n+1-k}&= (k+1+\alpha_1-\beta_3)_{n-k}\,(n+1+\alpha_1-\beta_3),\\
(\beta_3-\alpha_2)_{n+1-k} &= (\beta_3-\alpha_2)_{n-k}\,(n-k+\beta_3-\alpha_2),\\
(n+1+\alpha_1)_k &= (n+\alpha_1)_k\,\frac{n+k+\alpha_1}{n+\alpha_1}\,,\\
(n+2+\sigma)_k &= (n+1+\sigma)_k\,\frac{n+1+k+\sigma}{n+1+\sigma}\,,
\end{align*}
to write $A_{n+1, k}$ in the form
\begin{align*}
A_{n+1, k} &= A_{n, k} \left[\frac{(n+1)(n+1+\alpha_1-\beta_3)(n+1+\sigma)}{n+\alpha_1}
\right] g_n(k),\\
g_n(k) &=  \frac{(k+n+\alpha_1)(k-n+\alpha_2-\beta_3)}{(k-n-1)(k+n+1+\sigma)}\\
&= \frac{k^2 +\sigma k -(n+\alpha_1)(n+\beta_3-\alpha_2)}{k^2+\sigma k -(n+1)(n+1+\sigma)}.
\end{align*}

Regarding $k$ as a continuous variable as before, we differentiate
\begin{align*}
g_n'(k) = \frac{(2k+\sigma)(\beta_3-\alpha_2-1)(2n+1+\alpha_1)}
{\left[k^2+\sigma k -(n+1)(n+1+\sigma)\right]^2}.
\end{align*}
By the case assumption, it shows $\,g_n'(k)<0\,$ on the interval $[1, \infty)$.
In view of the limiting behavior $\,g_n(k)\to 1\,$ as $\,k\to\infty,\,$ hence,
we may conclude $\,g_n(k)>1\,$ for $\,k=1, \cdots, n,\,$
which leads to the estimate
\begin{equation}\label{A9}
A_{n+1, k} \ge A_{n, k} \left[\frac{(n+1)(n+1+\alpha_1-\beta_3)(n+1+\sigma)}{n+\alpha_1}
\right]
\end{equation}
for each $\,k=1, \cdots, n\,$ with strict inequalities when $\,k\ge 2.$

As for the initial term $A_{n+1, 0}$, we may write
\begin{align*}
A_{n+1, 0} = A_{n, 0} \bigl[(n+1+\alpha_1-\beta_3)(n+\beta_3-\alpha_2)].
\end{align*}
We observe that an upper bound for the last factor is given by
$$ n+\beta_3-\alpha_2<\frac{(n+1)(n+1+\sigma)}{n+\alpha_1},$$
which follows easily from the sign of cross difference
\begin{align*}
&(n+1)(n+1+\sigma) - (n+\beta_3-\alpha_2)(n+\alpha_1)\\
&\qquad= 2(1+\alpha_2-\beta_3) n + (1+\alpha_1)(1+\alpha_2-\beta_3)\\
&\qquad=(1+\alpha_2-\beta_3)(2n+1+\alpha_1) >0
\end{align*}
due to the case assumption. Since $\,A_{n, 0}<0,\,$ this upper bound gives
\begin{equation}\label{A10}
A_{n+1, 0} >A_{n, 0} \left[\frac{(n+1)(n+1+\alpha_1-\beta_3)(n+1+\sigma)}{n+\alpha_1}\right].
\end{equation}

Multiplying each term of \eqref{A9}, \eqref{A10} by $B_k$
and adding up, we obtain
\begin{align*}
\Omega_{n+1} &= A_{n+1, n+1} B_{n+1} + \sum_{k=0}^n A_{n+1, k} B_k\nonumber\\
& >  A_{n+1, n+1} B_{n+1} +  \left[\frac{(n+1)(n+1+\alpha_1-\beta_3)(n+1+\sigma)}{n+\alpha_1}\right]
\sum_{k=0}^n A_{n, k} B_k\nonumber\\
& = A_{n+1, n+1} B_{n+1} +  \left[\frac{(n+1)(n+1+\alpha_1-\beta_3)(n+1+\sigma)}{n+\alpha_1}\right]\Omega_n,
\end{align*}
which proves \eqref{A8} and our proof is now complete.
\end{proof}

\section{Rational extension of Newton diagram}
In this section we aim to extend the aforementioned positivity criterion of \cite{CY}
for the generalized hypergeometric functions of type \eqref{AG7}.

To state the criterion precisely, we recall that the Newton diagram associated to a finite set of planar points
$\,\bigl\{\left(\alpha_i,\,\beta_i\right) : i= 1, \cdots, m\bigr\}\,$
refers to the closed convex hull containing
$$\bigcup_{i=1}^m\,\Big\{(x, y)\in\R^2 : x\ge \alpha_i,\,\,y\ge \beta_i\Big\}\,.$$

For each $\,a>0,\,$ we denote by $O_a$ the set of $\,(b, c)\in\R_+^2\,$
defined by
\begin{align}
O_a &= \left\{ a<b<a+\frac 12,\,\, c\ge 3a+\frac 12-b\right\}\nonumber\\
&\qquad \cup
\left\{ a<c<a+\frac 12,\,\, b\ge 3a+\frac 12-c\right\}
\quad\text{if}\quad a\ge\frac 12\,,\label{N1}\\
O_a &= \left\{ a<b<2a,\,\, c\ge 3a+\frac 12-b\right\}\nonumber\\
&\qquad \cup
\left\{ a<c<2a,\,\, b\ge 3a+\frac 12-c\right\}
\quad\text{if}\quad 0<a<\frac 12\,,\label{N2}
\end{align}
which represents two symmetric infinite strips bounded by
$\,b+c= 3a + 1/2\,$ and four half-lines parallel to the coordinate axes.

By combining the methods of Fields and Ismail \cite{FI}, Gasper \cite{Ga1}
and fractional integrals with the squares of Bessel functions as kernels,
two of the present authors established the following criterion.

\smallskip

\begin{theorem}\label{theoremN1}{\rm(Cho and Yun, \cite{CY})}
For $\,a>0,\,b>0,\,c>0,\,$ put
$$\Phi(x)=\F\left[\begin{array}{c}
a\\ b, c\end{array}\biggr| -\frac{\,x^2}{4}\right]\qquad(x>0).$$
Let $P_a$ be the Newton diagram associated to
$\,\Lambda = \left\{\left(a+\frac 12,\,2a\right),\,\left(2a,\,a+\frac 12\right)\right\},\,$
$O_a$ the set defined in \eqref{N1}, \eqref{N2} and $N_a$ the complement of $\,P_a\cup O_a\,$ in $\R_+^2$ so that
the decomposition $\,\R_+^2 = P_a\cup O_a\cup N_a\,$ holds.
\begin{itemize}
\item[\rm(i)] If $\,\Phi\ge 0,\,$ then necessarily
\begin{equation}\label{ND1}
b>a,\,\,c>a,\,\,b+c\ge 3a+\frac 12\,.
\end{equation}
\item[\rm(ii)] If $\,(b, c)\in P_a\,,$ then $\,\Phi\ge 0\,$ and strict positivity holds unless $(b, c)\in\Lambda.$
\item[\rm(iii)] If $\,(b, c)\in N_a\,,$ then $\Phi$ alternates in sign.
\end{itemize}
\end{theorem}

\smallskip

\begin{remark}\label{remarkSF}
For the cases of nonnegativity, we follow \cite{CY} to introduce
\begin{align}\label{SF2}
\J_\alpha(x) = {}_0F_1\left(\alpha +1\,; -\frac{\,x^2}{4}\right)\qquad(\alpha>-1).
\end{align}
Owing to the relation of \eqref{SF1},
it is easy to see that $\J_\alpha$ shares positive zeros in common with
Bessel function $J_\alpha$ and its square takes the form
\begin{align}\label{SF3}
\mathbb{J}_\alpha^2\left(x\right) =
\F\left[\begin{array}{c}
\alpha + \frac 12\\ \alpha+1, 2\alpha+1\end{array}\biggr| - x^2\right]
\end{align}
when $\,\alpha>-1/2.\,$ Consequently, if $\,(b, c)\in\Lambda,\,$ then
\begin{equation}\label{SF3}
\Phi(x) = \F\left[\begin{array}{c}
a\\ a+\frac 12, 2a\end{array}\biggr| -\frac{\,x^2}{4}\right] =
\mathbb{J}_{a-\frac 12}^2\left(\frac x2\right),
\end{equation}
which is nonnegative but has infinitely many zeros on $(0, \infty)$.
\end{remark}

\smallskip

Theorem \ref{theoremN1} unifies many of earlier positivity results and we refer to our recent paper \cite{CCY}
in which it is applied to improve the results of Misiewicz and Richards \cite{MR}, Buhmann \cite{Bu} at the same time.

For $\,(b, c)\in O_a,\,$ it is left undetermined whether positivity holds or not.
We now state our main extension theorem which still does not fill out the whole of $O_a$
but covers the upper half of the rational function
$$ c= a+\frac{a}{2(b-a)}\qquad(b>a).$$
We shall use the letter $\Lambda$ below for the same notation as above.

\smallskip

\begin{theorem}\label{theoremN2}
For $\,a>0,\,b>0,\,c>0,\,$ put
$$\Phi(x)=\F\left[\begin{array}{c}
a\\ b, c\end{array}\biggr| -\frac{\,x^2}{4}\right]\qquad(x>0).$$
Let $P_a^*$ be the set of parameter pairs $\,(b, c)\in\R_+^2\,$ defined by
\begin{align*}
P_a^* = \left\{b>a,\,c>a,\,
c\ge\max\Big[ 3a+\frac 12-b,\,\,a+ \frac{a}{2(b-a)}\Big]\right\}.
\end{align*}
If $\,(b, c)\in P_a^*\setminus\Lambda,\,$ then $\Phi$ is strictly positive.
\end{theorem}

\begin{remark}
For the sake of convenience, we illustrate Theorems \ref{theoremN1}, \ref{theoremN2}
with Figures \ref{Fig2}, \ref{Fig3} for the case $\,a>1/2,\,a=1/2,\,$
separately (the case $\,0<a<1/2\,$ is similar to the case $\,a>1/2$). In each figure,
the red-colored part indicates the improved positivity region $P_a^*$ and
the grey-colored part indicates the region where positivity breaks down. The blank or
white-colored parts indicate the missing region.
\end{remark}

\begin{figure}[!h]
 \centering
 \includegraphics[width=280pt, height= 200pt]{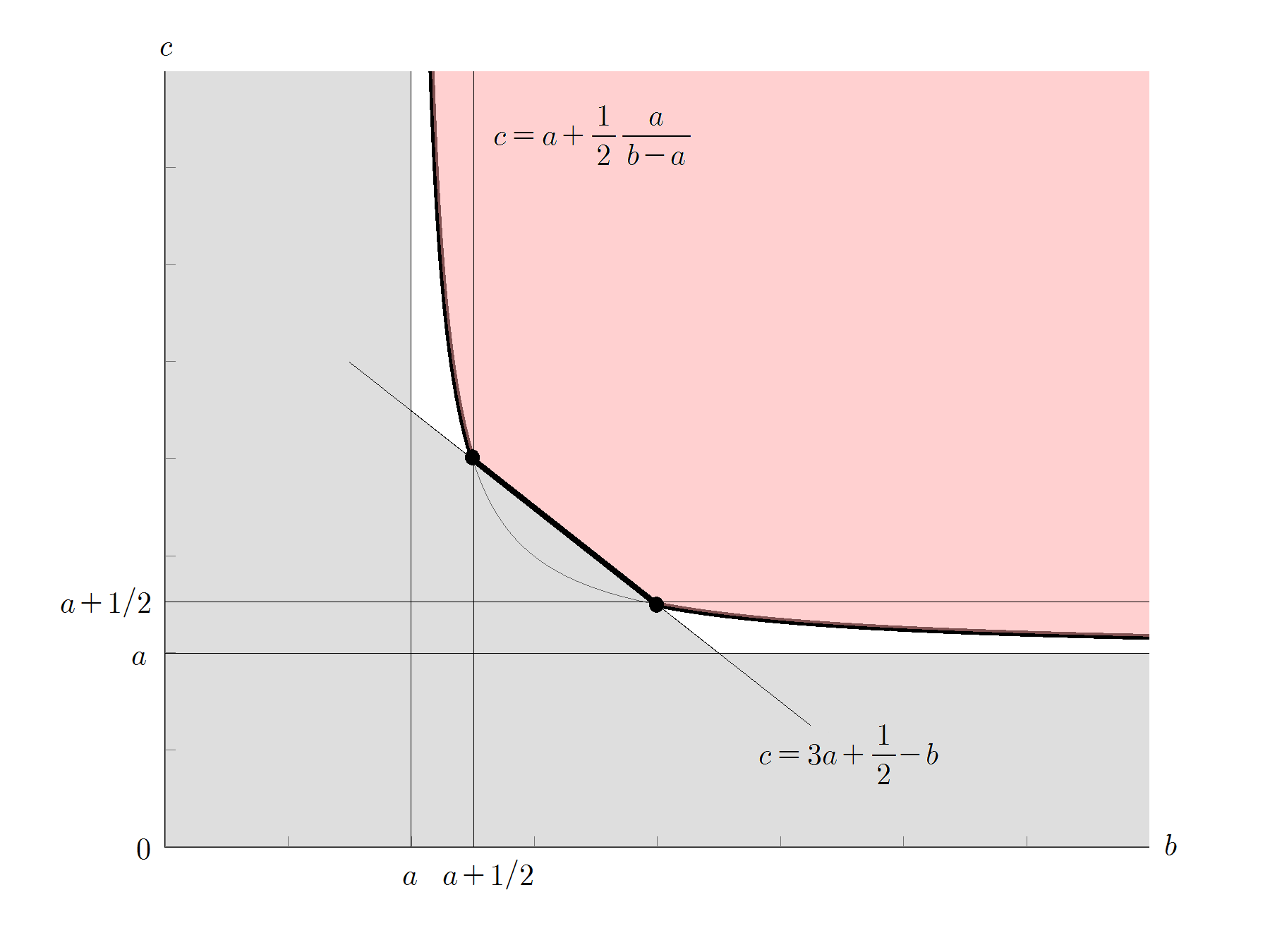}
 \caption{The improved positivity region $P_a^*$ in the case $\,a>\frac 12\,$ which includes
 the Newton diagram associated to $\,\Lambda=\{(a+1/2, 2a), \,(2a, a+1/2)\}\,$.}
\label{Fig2}
\end{figure}

\bigskip

\begin{figure}[!h]
 \centering
 \includegraphics[width=280pt, height=200pt]{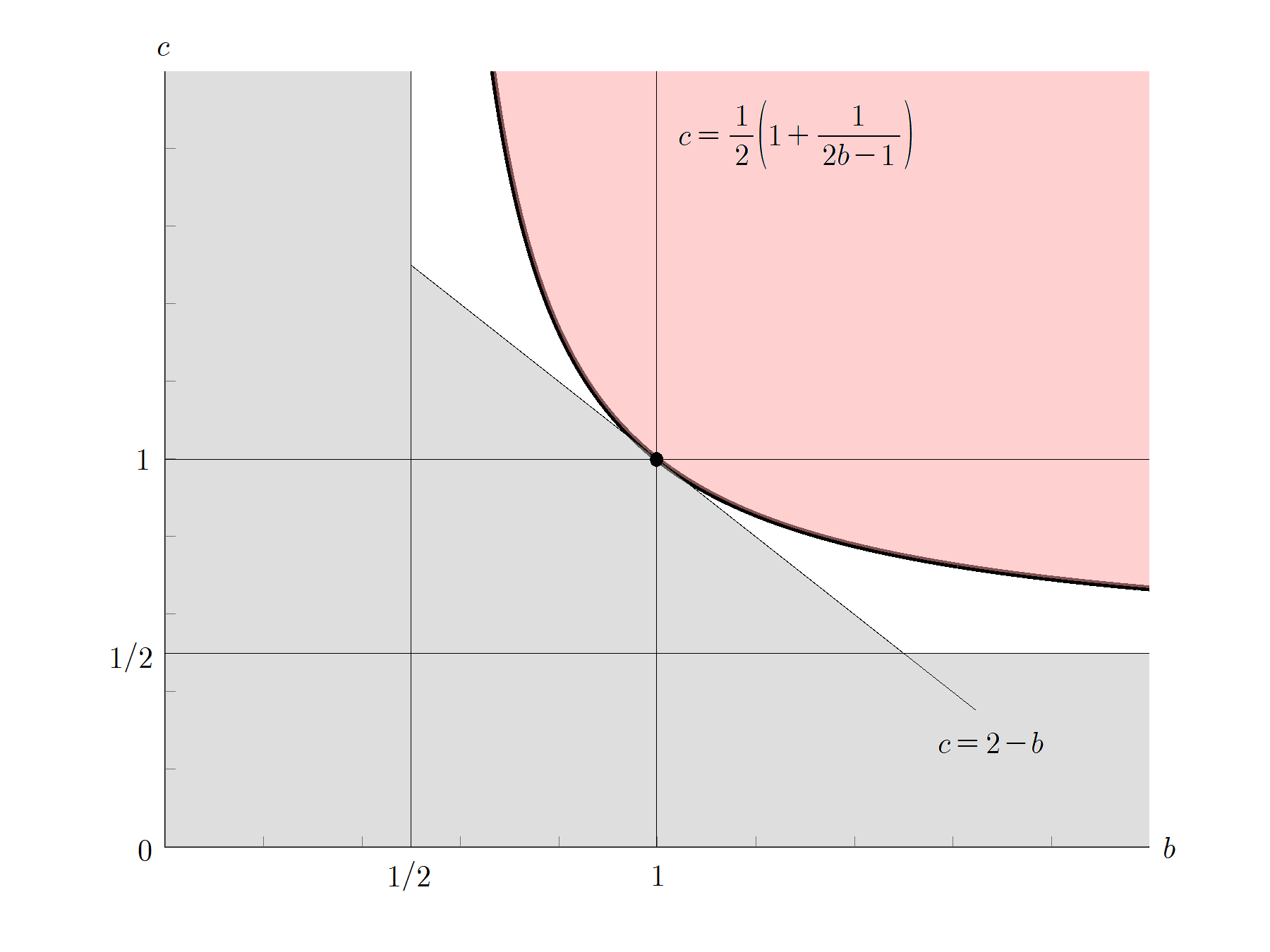}
 \caption{The improved positivity region $P_{\frac 12}^*$ which includes
the Newton diagram associated to $\,\Lambda=\{(1, 1)\}.$}
\label{Fig3}
\end{figure}

\paragraph{Proof of Theorem \ref{theoremN2}.}
In view of the difference
$$ a+ \frac{a}{2(b-a)} - \left[ 3a+\frac 12-b\right] = \frac{\left(b-a-\frac 12\right)(b-2a)}{b-a}\,,$$
it is graphically obvious that the rational function $\,c= a+ a/2(b-a)\,$ lies below the line
$\,c= 3a+1/2-b\,$ only for $\,b\in L,\,$ where $L$ denotes
$$L=\left\{(1-t)(a+1/2) + t(2a) : 0\le t\le 1\right\}.$$

As it is already shown in Theorem \ref{theoremN1} that $\Phi$ is strictly positive for $\,(b, c)\in P_a\setminus\Lambda,\,$ it remains to
prove the positivity of $\Phi$ in the case $\,(b, c)\in P_a^*\,$ with $b$ lying outside
the closed interval $L$. By symmetry in $b, c,$ we may assume $\,b\le c\,$ and hence it suffices to
deal with the case
\begin{equation}\label{W1}
c\ge a+ \frac{a}{2(b-a)},
\end{equation}
where $\,a<b<a+1/2\,$ when $\,a\ge 1/2\,$ or $\,a<b<2a\,$ when $\,0<a<1/2.$

We apply Gasper's
{\it sums of squares formula} (\cite{Ga1}, (3.1)) to write
\begin{align}\label{W2}
\Phi(x) &= \Gamma^2(\nu+1)\left( \frac{x}{4}\right)^{-2\nu}\biggl\{J_{\nu}^2\left( \frac{x}{2}\right) +\nonumber\\
&\qquad \sum_{n=1}^{\infty} C(n, \nu)\frac{2n+2\nu}{n+2\nu}\frac{(2\nu+1)_n}{n!}
J_{\nu+n}^2\left( \frac{x}{2}\right)\biggr\}
\end{align}
in which $C(n, \nu)$ denotes the terminating series defined by
\begin{equation}\label{W3}
C(n, \nu)= {}_4F_{3} \left[\begin{array}{c} -n,\,n+2\nu,\,\nu+ 1,\,a\\
\nu+\frac 12,\,b,\,c\end{array}\right]
\end{equation}
and $\nu$ can be arbitrary as long as $2\nu$ is not a negative integer.

Due to the interlacing property on the zeros of Bessel functions $\,J_\nu,\,J_{\nu+1}\,$
(see Watson \cite{Wa}), the positivity of $\Phi$ would follow instantly from formula \eqref{W2} if
$\,C(n, \nu)>0\,$ for all $n$, for instance, and $\,\nu>-1/2.$

To investigate the sign of $C(n, \nu)$, we apply Lemma \ref{lemmaA2} with
$$\alpha_1=2\nu, \,\,\alpha_2=\nu+1,\,\,\alpha_3=a,\,\,\beta_1=\nu+\frac 12,\,\,\beta_2=b,\,\,\beta_3=c.$$
The Saalsch\"utzian condition (A1) of \eqref{A} is equivalent to the choice
\begin{equation}\label{W4}
\nu = \frac 12\left(b+ c-a-\frac 32\right).
\end{equation}
It is elementary to translate conditions (A2), (A3), (A4) of \eqref{A} into
\begin{equation}\label{W5}
\left\{\begin{aligned} &{\,\, c > b-a + \frac 12,\quad b\ge a-\frac 12\,,}\\
&{\,\,c >3a+\frac 12-b,\quad b>a,}\\
&{\,\,c\ge a+ \frac{a}{2(b-a)}}.
\end{aligned}\right.
\end{equation}

On inspecting the region determined by \eqref{W5} in the $(b, c)$-plane, it is immediate to
find that \eqref{W5} amounts to \eqref{W1} subject to the restriction
$\,a<b<a+1/2\,$ when $\,a\ge 1/2\,$ or $\,a<b<2a\,$ when $\,0<a<1/2.$

By Lemma \ref{lemmaA2}, we may conclude $\,C(1, \nu)\ge 0\,$ and $\,C(n, \nu)>0\,$ for all $\,n\ge 2\,$
with $\nu$ chosen according to \eqref{W4} and our proof is now complete.\qed

\section{Askey-Szeg\"o problem}
Returning to problem \eqref{AG1}, an application of the above positivity criteria
in an obvious way yields what we aimed to establish.

\smallskip

\begin{theorem}\label{theoremN3}
For $\,\alpha>-1,\,\beta<\alpha+1,\,$ put
$$\psi(x) = \int_0^x t^{-\beta} J_\alpha(t) dt\qquad(x>0).$$
\begin{itemize}
\item[\rm(i)] If $\,\psi\ge 0,\,$ then necessarily
\begin{equation*}
-\alpha-1<\beta<\alpha+1,\,\,\beta\ge -\frac 12\,.
\end{equation*}
\item[\rm(ii)] Let $\mathcal{P}^*$ be the set of $\,(\alpha, \beta)\in\R^2\,$ defined by
\begin{align*}
\mathcal{P}^*=\left\{\alpha>-1, \,\,\max\Big[-\frac 12,\,-\frac 13 (\alpha+1)\Big]\le\beta<\alpha+1\right\}
\end{align*}
Then $\,\psi\ge 0\,$ for each $\,(\alpha, \beta)\in \mathcal{P}^*\,$ and strict positivity holds
unless it coincides with $\,(1/2, -1/2).$
\end{itemize}
\end{theorem}

\begin{figure}[!h]
 \centering
 \includegraphics[width=300pt, height=220pt]{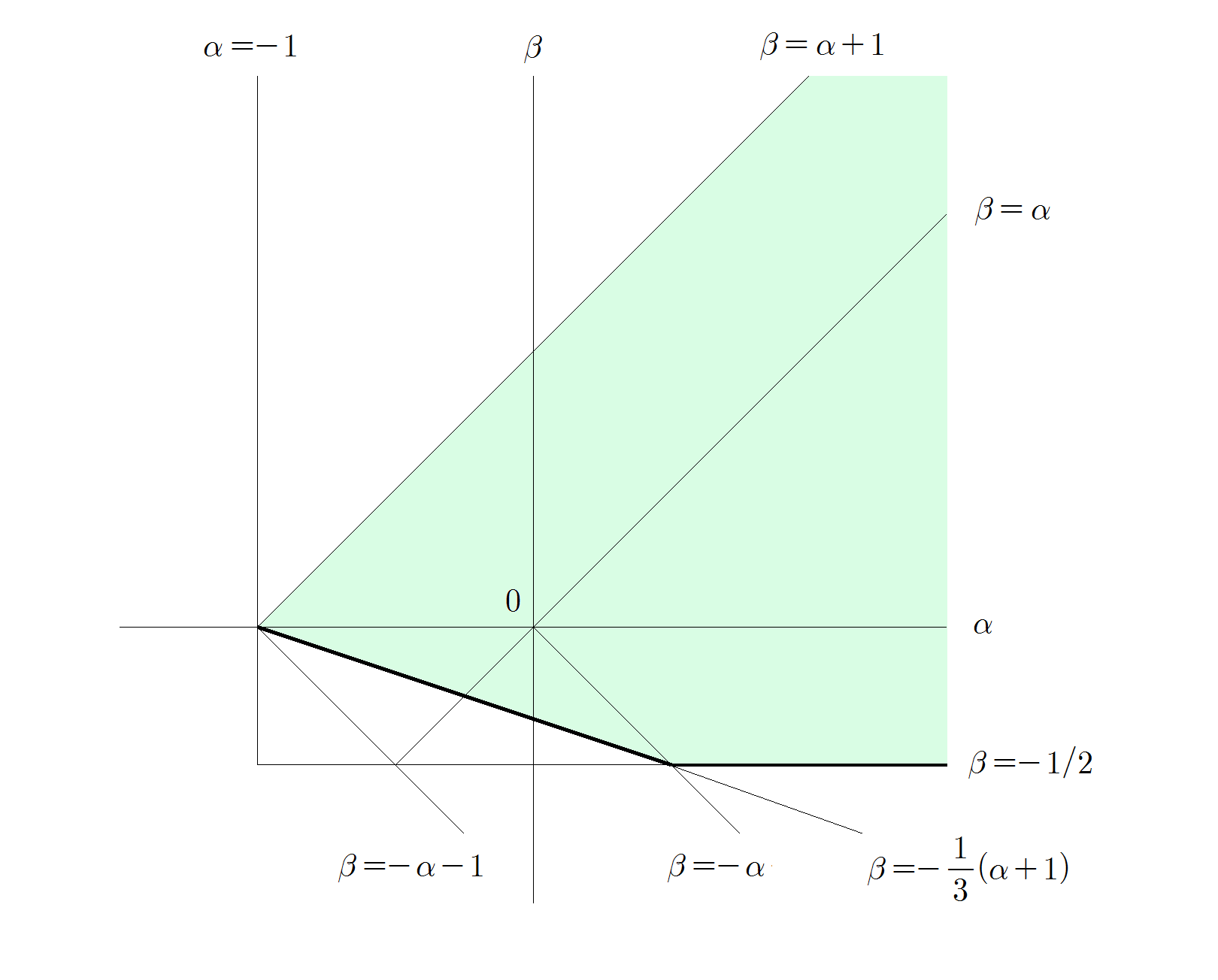}
 \caption{The improved positivity region for problem \eqref{AG1} in which the line
 $\,\beta=\alpha\,$ corresponds to Szeg\"o's problem \eqref{AG4}.}
\label{Fig4}
\end{figure}

\begin{remark} In Figure \ref{Fig4}, the green-colored part represents
$\mathcal{P}^*$. As it is evident pictorially on comparing with Figure \ref{Fig1},
Theorem \ref{theoremN3} improves Theorem A by adding
the triangle with boundary lines
$$\beta= 0,\,\,\beta=-\alpha,\,\,\beta=-\frac 13(\alpha+1)$$
as a new positivity region and by narrowing down the necessity region.
\end{remark}

\paragraph{Proof of Theorem \ref{theoremN3}.}
In view of \eqref{AG6}, it suffices to deal with
\begin{equation}\label{IAG1}
\Psi(x) = \F\left[\begin{array}{c} \frac{\alpha-\beta+1}{2}\\ \alpha+1,
\frac{\alpha-\beta+1}{2} +1\end{array} \biggl| -\frac{\,x^2}{4}\right]\qquad(x>0).
\end{equation}

Under the assumption $\,\alpha>-1,\,\beta<\alpha+1,\,$ each parameter of $\Psi$ is positive.
If $\,\Psi\ge 0,\,$ then it follows from necessary condition \eqref{ND1} that
\begin{align*}
\alpha+1 &> \frac{\alpha-\beta+1}{2},\\
\frac{\alpha-\beta+1}{2} +1 &\ge 3\left( \frac{\alpha-\beta+1}{2}\right)
+\frac 12 -(\alpha+1),
\end{align*}
which reduces to the stated necessary condition of part (i).

To prove part (ii), we apply Theorem \ref{theoremN2} with
$$ a=  \frac{\alpha-\beta+1}{2}\,,\,\,b=\alpha+1,\,\,c=  \frac{\alpha-\beta+1}{2} +1.$$
Inspecting the conditions $\,c\ge 3a+1/2-b,\,c\ge a+a/2(b-a)\,$ for $\,b>a\,$ in terms of $\alpha, \beta$ separately, it is
elementary to find the condition
$$c\ge\max\Big[ 3a+\frac 12-b,\,\,a+ \frac{a}{2(b-a)}\Big],\,\,b>a,$$
is equivalent to
\begin{equation}\label{P1}
\beta\ge \max\Big[-\frac 12,\,-\frac 13 (\alpha+1)\Big],\,\,\beta>-\alpha-1.
\end{equation}

Combining \eqref{P1} with the necessary condition of part (i), we deduce $\,\Psi\ge 0\,$
for each $\,(\alpha, \beta)\in\mathcal{P}^*.$ Regarding strict positivity, we note that
the nonnegativity condition required by
$$(b, c)\in\Lambda = \left\{\Big(a+\frac 12, 2a\Big),\,\,\Big(2a, a+\frac 12\Big)\right\}$$
reduces to the single case $\,(\alpha, \beta)= (1/2, -1/2).\,$ Indeed,
\begin{equation}
\Psi(x) = \F\left(1\,; \frac 32, \,2\,; -\frac{\,x^2}{4}\right)
=\left[\frac{\sin(x/2)}{x/2}\right]^2
\end{equation}
in this case, which is nonnegative but has infinitely many positive zeros.

By Theorem \ref{theoremN2}, we conclude
$\Psi$ is strictly positive for each $\,(\alpha, \beta)\in\mathcal{P}^*\,$ unless it coincides with
$\,(1/2, -1/2)\,$ and our proof is complete.\qed

\medskip

As an immediate consequence of Theorem \ref{theoremN3}, we obtain the following
upper and lower bounds for $\beta(\alpha)$ and $\bar{\alpha}.$

\smallskip

\begin{corollary}\label{corollaryN1} Under the same setting as in Theorem B, we have
\begin{align*}
{\rm(i)}\quad &\,\max\left(-\alpha-1, \,-\frac 12\right)<\beta(\alpha)\le -\frac 13 (\alpha+1),\\
{\rm(ii)}\quad &\,\,\,\lim_{\alpha\to -1 +}\beta(\alpha) = 0,\quad \beta\left(\frac 12\right) = -\frac 12\,,\\
{\rm(iii)}\quad & \,\,\,-\frac 12<\bar{\alpha}\le -\frac 14\,.
\end{align*}
\end{corollary}

\smallskip

\begin{remark} While the results are evident by Theorem B and Theorem \ref{theoremN3},
that $\,\beta(1/2)= -1/2\,$ can be verified in a simple way. Indeed, the formula
$$J_{\frac 12}(t)= \sqrt{\frac {2t}{\pi}}\, \frac{\sin t}{t}$$
(see Luke \cite{L}, Watson \cite{Wa}) implies that $\,j_{\frac 12,\,2} = 2\pi\,$ and
$$\int_0^{2\pi} \sqrt t J_{\frac 12}(t) dt = \sqrt{\frac {2}{\pi}}\,\int_0^{2\pi} \sin t\, dt =0,$$
whence the desired value follows instantly by the uniqueness of $\beta(\alpha)$.
\end{remark}

\smallskip

Corollary \ref{corollaryN1} indicates that $\,\beta=\beta(\alpha),\,-1<\alpha\le 1/2,\,$ is a smooth
curve joining $\,(-1, 0),\,(1/2, -1/2)\,$ which lies in the triangle
determined by
$$\beta=-\alpha-1,\quad\beta=-1/2,\quad \beta=-(\alpha+1)/3.$$

In \cite{AS}, Askey and Steinig gave a list of numerical approximations for $\beta(\alpha)$.
To get an insight into how accurate or informative the above upper bound would be, we compare it with
their list as follows.

\bigskip

\centerline{
\begin{tabular}{c|c|c}
  %\hline
  % after \\: \hline or \cline{col1-col2} \cline{col3-col4} ...
  $\quad\alpha\quad$ & $\qquad\beta(\alpha)\qquad$ &$\qquad-\dfrac 13 (\alpha+1)\qquad$\\\hline
  $ -0.5 $ & $-0.1915562$ &$-0.1666667 $\\
$-0.4$ & $-0.2259427$ & $-0.2000000$\\
$-0.3$ & $-0.2593436$ & $-0.2333333$\\
$-0.2$ & $-0.2918541$ & $-0.2666667$\\
$-0.1$ & $-0.3235531$ & $-0.3000000$\\
$\,\,\,0$ & $-0.3545096$& $-0.3333333$\\
$\,\,\,\,0.1$ & $-0.3847832$ & $-0.3666667$\\
$\,\,\,\,0.2$ & $-0.4144258$ & $-0.4000000$\\
$\,\,\,\,0.3$ & $-0.4434834$ & $-0.4333333$\\
$\,\,\,\,0.4$ & $-0.4719960$ & $-0.4666667$
\end{tabular}}

\bigskip

Regarding the approximated values as true ones, these comparisons show that $\beta(\alpha)$ lies
within distance $0.026$ from $\,-(\alpha+1)/3\,$ and the error increases
up to certain point near $\,\alpha = -0.3\,$ and then decreases to zero.

On the other hand, we also point out that G. Szeg\"o \cite{Fe} approximated $\,\bar{\alpha}\approx -0.2693885,\,$
whereas our upper bound of $\bar{\alpha}$ is $\,-0. 25.$

\section{Gasper's extensions}
As a generalization of \eqref{AG1}, we consider the problem of determining
parameters $\,\alpha, \beta, \gamma \,$ for the inequality
\begin{equation}\label{GE1}
\int_0^x (x^2- t^2)^\gamma t^{-\beta}J_\alpha(t) dt \ge 0\qquad(x>0),
\end{equation}
which reduces to problem \eqref{AG1} in the special case $\,\gamma =0.$

By integrating termwise, it is plain to evaluate
\begin{align}\label{GE2}
\int_0^x (x^2- t^2)^\gamma t^{-\beta}J_\alpha(t) dt &=
\frac{B\left(\gamma+1, \frac{\alpha-\beta+1}{2}\right)}
{2^{\alpha+1}\Gamma(\alpha+1)}\,x^{\alpha-\beta+2\gamma+1}
\nonumber\\
&\quad\times\, \F\left[\begin{array}{c} \frac{\alpha-\beta +1}{2}\\ \alpha+1,
\frac{\alpha-\beta+1}{2} +\gamma+1\end{array} \biggl| -\frac{\,x^2}{4}\right]
\end{align}
subject to the condition $\,\alpha>-1,\,\gamma>-1,\,\beta<\alpha+1,\,$
where $B$ denotes the usual Euler's beta function. In analogy with \eqref{AG1},
hence, problem \eqref{GE1} is equivalent to the positivity question on the $\F$
generalized hypergeometric function defined on the right side of \eqref{GE2}.

In \cite{Ga1}, Gasper employed the sums of squares method and
an interpolation argument involving fractional integrals to prove that \eqref{GE1} holds with
strict positivity for each $\,(\alpha, \beta)\in\mathcal{S}_\gamma\setminus\left\{ (\gamma+1/2, \,-\gamma-1/2)\right\},\,$
where
\begin{equation}
\mathcal{S}_\gamma = \left\{\alpha\ge \gamma+\frac 12,\,\, \alpha-2\gamma-1\le\beta<\alpha+1\right\}
\end{equation}
in the case $\,-1<\gamma\le -1/2\,$ and
\begin{align}
\mathcal{S}_\gamma &=
\bigl\{\alpha>-1,\, \,0\le\beta<\alpha+1\bigr\}\nonumber\\
&\qquad\cup\,\left\{\alpha\ge\gamma+\frac 12,\,\,-\Big(\gamma+\frac 12\Big)\le\beta\le 0\right\}
\end{align}
in the case $\,\gamma>-1/2\,$ (see Figures \ref{Fig5}, \ref{Fig6} below).

Our purpose here is to improve Gasper's result as follows.

\smallskip

\begin{theorem}\label{theoremN4}
Let $\,\alpha>-1,\,\gamma>-1,\,\beta<\alpha+1.$
\begin{itemize}
\item[\rm(i)] If \eqref{GE1} holds, then necessarily
$$ \beta\ge -\Big(\gamma+\frac 12\Big),\,\,-\alpha-1<\beta<\alpha+1.$$
\item[\rm(ii)] For each $\,\gamma>-1,\,$ let $\,\mathcal{S}_\gamma^*\,$ be the set of
$\,(\alpha, \beta)\in\R^2\,$ defined by
\begin{equation*}
\mathcal{S}_\gamma^* = \left\{\,\alpha>-1,\,\,
\max\left[-\Big(\gamma+\frac 12\Big),\,\,
-\,\frac{2\gamma+1}{2\gamma+3} (\alpha+1)\right] \le\beta<\alpha+1\,\right\}.
\end{equation*}
If $\,(\alpha, \beta)\in\mathcal{S}_\gamma^*,\,$ then \eqref{GE1} holds with strict positivity
unless
$$\alpha=\gamma+\frac 12, \,\beta= -\Big(\gamma+\frac 12\Big)
\quad\text{or}\quad \gamma=-\frac 12\,,\,\,\beta=0.$$
\end{itemize}
\end{theorem}

\smallskip

\paragraph{Proof of Theorem \ref{theoremN4}.}
In view of \eqref{GE2}, it suffices to deal with
\begin{equation}
\Sigma(x) = \F\left[\begin{array}{c} \frac{\alpha-\beta+1}{2}\\ \alpha+1,
\frac{\alpha-\beta+1}{2} +\gamma+1\end{array} \biggl| -\frac{\,x^2}{4}\right]\qquad(x>0).
\end{equation}

On setting
$$ a=  \frac{\alpha-\beta+1}{2}\,,\,\,b=\alpha+1,\,\,c=  \frac{\alpha-\beta+1}{2} +\gamma+ 1$$
and applying the necessity part of Theorem \ref{theoremN1}, Theorem \ref{theoremN2} in the same way as in the proof of
Theorem \ref{theoremN3}, it is straightforward to verify (i), (ii).

As for the cases of nonnegativity, we note from \eqref{SF3} of Remark \ref{remarkSF} that
if $\,\alpha=\gamma+1/2,\,\beta=-(\gamma+1/2),\,\gamma>-1, \,$ then
\begin{equation}
\Sigma(x) = \F\left[\begin{array}{c} \gamma+1\\ \gamma+\frac 32,
2(\gamma+1)\end{array} \biggl| -\frac{\,x^2}{4}\right] = \J_{\gamma+\frac 12}^2\left(\frac x2\right).
\end{equation}
On the other hand, if $\,\alpha>-1, \,\beta=0,\, \gamma=-1/2,\,$ then
\begin{equation}
\Sigma(x) = \F\left[\begin{array}{c} \frac{\alpha+1}{2}\\ \alpha+1,
\frac{\alpha+2}{2}\end{array} \biggl| -\frac{\,x^2}{4}\right] = \J_{\frac\alpha 2}^2\left(\frac x2\right).
\end{equation}
Both identities show $\,\Sigma\ge 0\,$ with infinitely many positive zeros.\qed

\begin{figure}[!h]
 \centering
 \includegraphics[width=300pt, height=200pt]{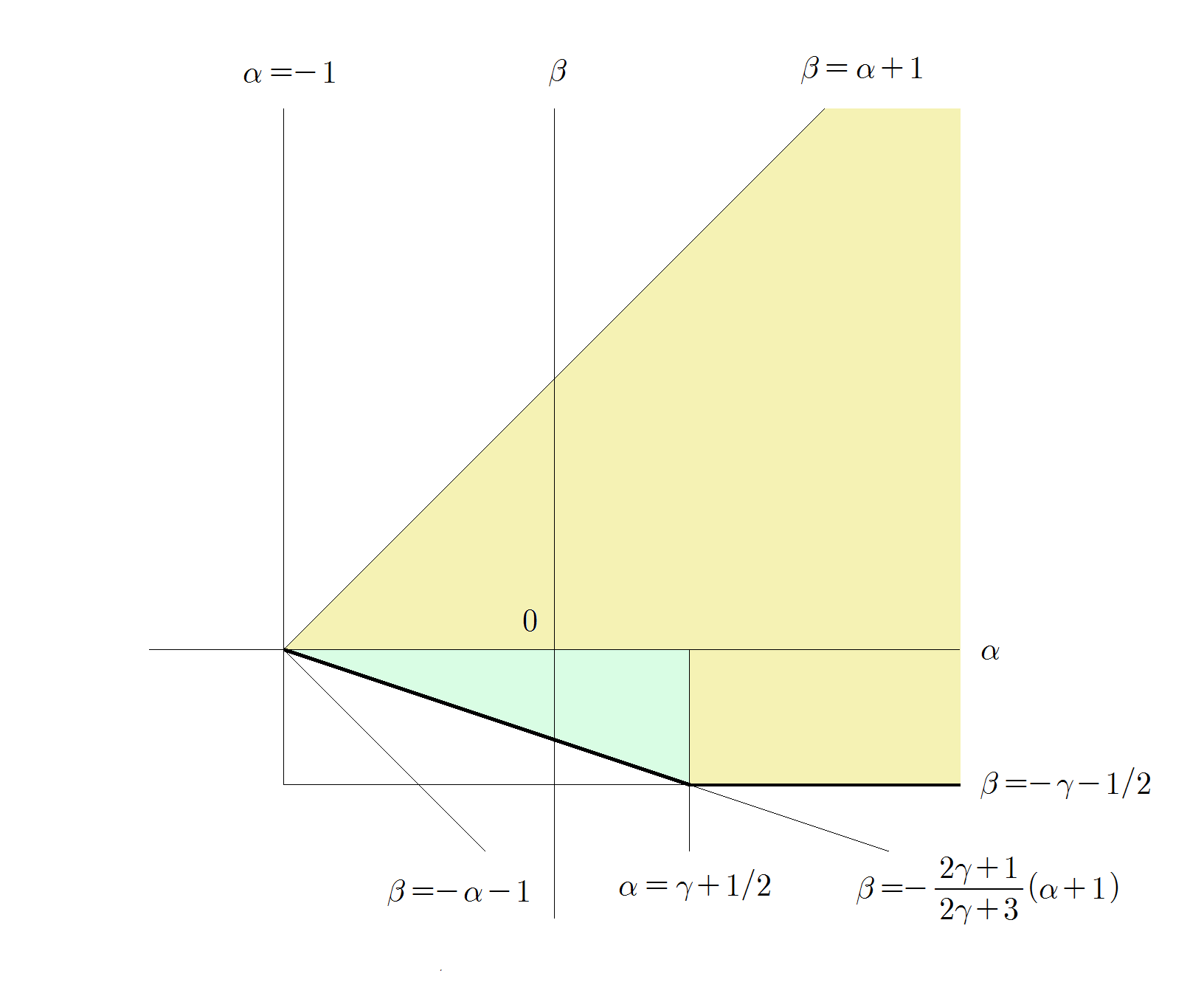}
 \caption{The improved positivity region for problem \eqref{GE1} in the case $\,\gamma>-1/2\,,$ where
 the yellow-colored part represents Gasper's region.}
\label{Fig5}
\end{figure}

\bigskip

\begin{figure}[!h]
 \centering
 \includegraphics[width=300pt, height=220pt]{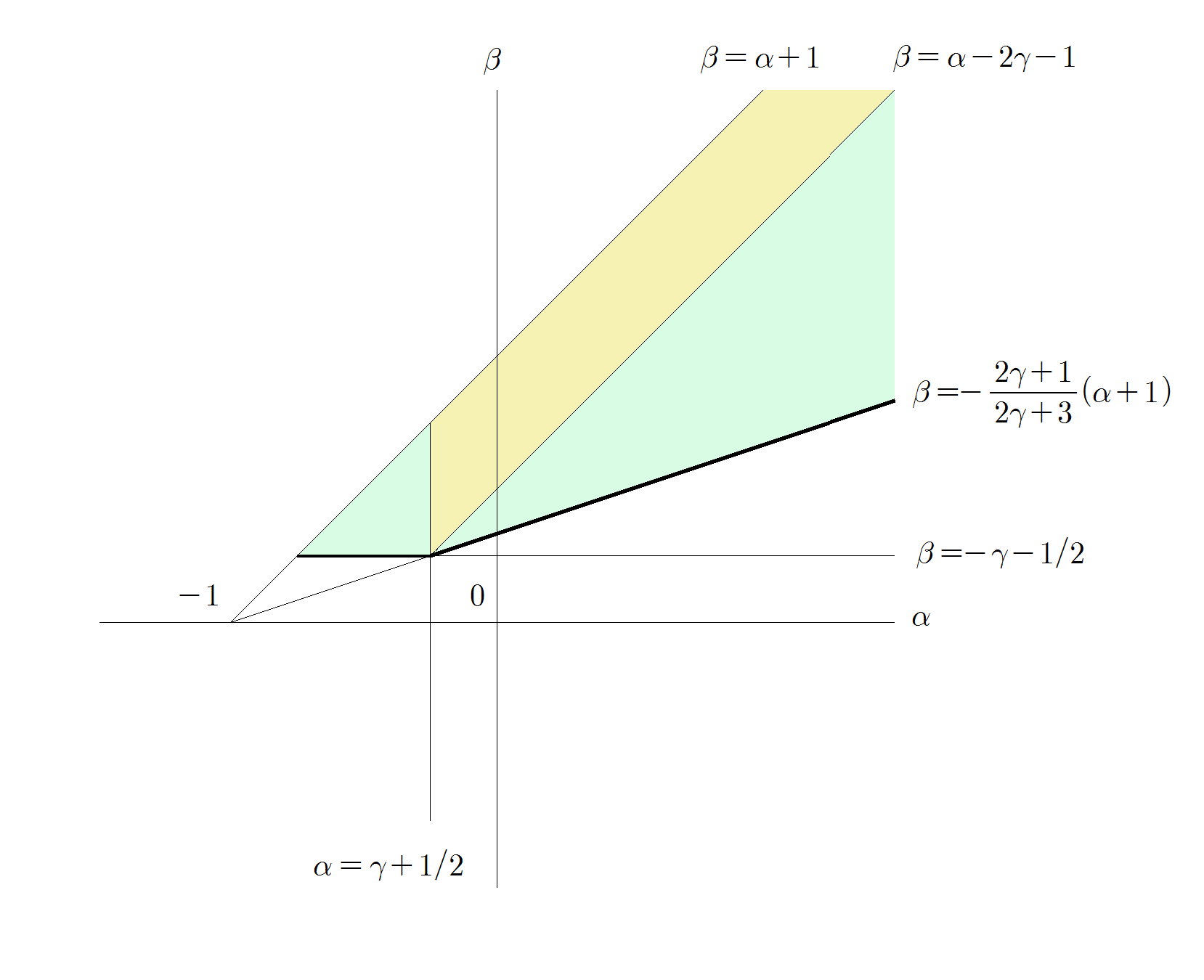}
 \caption{The improved positivity region for problem \eqref{GE1} in the case $\,-1<\gamma< -1/2,\,$ where
  the yellow-colored part represents Gasper's region.}
\label{Fig6}
\end{figure}

\begin{remark} As for the missing ranges, we point out the following:
\begin{itemize}
\item In the case $\,\gamma>-1/2,\,$ as shown in Figure \ref{Fig5}, Theorem \ref{theoremN4} leaves the triangle
formed by the boundary lines
$$\beta=-\alpha-1, \,\,\beta= -(\gamma+1/2),\,\,\beta= -\,\frac{2\gamma+1}{2\gamma+3} (\alpha+1)$$
and it is an open question if it is possible to give a necessary and sufficient condition
in terms of certain transcendental root $\beta_\gamma(\alpha)$ in an analogous manner with the case
$\,\gamma=0.$
\item In the case $\,\gamma=-1/2,\,$ problem \eqref{GE1} is completely resolved in the sense that it holds if
and only if $\,\alpha>-1, \,0\le\beta<\alpha+1.$
\item In the case $\,-1<\gamma<-1/2,\,$ as shown in Figure \ref{Fig6}, Theorem \ref{theoremN4} leaves
the infinite sector defined by
$$\alpha> \gamma+1/2,\, \,\,-(\gamma+1/2)\le\beta<-\,\frac{2\gamma+1}{2\gamma+3} (\alpha+1).$$
\end{itemize}
\end{remark}

\bigskip

{{\large \bf Acknowledgements.}} Yong-Kum Cho is supported by the National Research Foundation of Korea Grant
funded by the Korean Government \# 20160925. Seok-Young Chung is supported by the Chung-Ang University
Excellent Student Scholarship in 2017. Hera Yun is supported by the Chung-Ang University
Research Scholarship Grants in 2014--2015.

\bigskip

\noindent
Yong-Kum Cho (ykcho@cau.ac.kr), Seok-Young Chung (sychung@cau.ac.kr) and Hera Yun (herayun06@gmail.com)

\medskip

\noindent
Department of Mathematics, College of Natural Science,
Chung-Ang University, 84 Heukseok-Ro, Dongjak-Gu, Seoul 06974, Korea

\end{document}